\documentclass[12pt]{amsart} 
\usepackage[active]{srcltx} 
\usepackage[margin=1in]{geometry} 
\usepackage{amsthm,amsfonts,amsmath,amssymb} 
\usepackage{dsfont} 
\usepackage[T1]{fontenc} 
\usepackage[utf8]{inputenc} 
\usepackage{fourier} 
\usepackage[dvipsnames,usenames]{color}

\newcommand{\D}{\mathbb{D}} 
 
\newcommand{\C}{\mathbb{C}} 
\newcommand{\N}{\mathbb{N}} 
 
\newcommand{\T}{\mathbb{T}}

\newcommand{\Hi}{{\mathcal{H}}^\infty} 
\newcommand{\Ho}{{\mathcal{H}}^1} 
\newcommand{\Ht}{{\mathcal{H}}^2}

\newcommand{\Hp}{{\mathcal{H}}^p}

\newcommand{\fin}{{\operatorname{fin}}}

\theoremstyle{theorem} 
\newtheorem{theorem}{Theorem}[section] 
\newtheorem{lemma}[theorem]{Lemma} 
\newtheorem{corollary}[theorem]{Corollary}

\theoremstyle{definition} 
\newtheorem*{definition}{Definition}

\theoremstyle{remark}

\newtheorem*{conj}{Conjecture}

\DeclareMathOperator{\mre}{Re} \DeclareMathOperator{\mim}{Im}

\begin{document}

\author{Andriy Bondarenko}

\address{Andriy Bondarenko\\
Department of Mathematical Sciences \\
Norwegian University of Science and Technology \\
NO-7491 Trondheim \\
Norway} \email{andriybond@gmail.com}

\author{Ole Fredrik Brevig} \address{Ole Fredrik Brevig \\
Department of Mathematical Sciences \\
Norwegian University of Science and Technology \\
NO-7491 Trondheim \\
Norway} \email{ole.brevig@math.ntnu.no}

\author{Eero Saksman} \address{Eero Saksman \\ Department of Mathematical Sciences \\
Norwegian University of Science and Technology \\
NO-7491 Trondheim \\ Norway \\ and
Department of Mathematics and Statistics \\
University of Helsinki \\
FI-00170 Helsinki \\
Finland} \email{eero.saksman@helsinki.fi}

\author{Kristian Seip} \address{Kristian Seip\\Department of Mathematical Sciences \\
Norwegian University of Science and Technology \\
NO-7491 Trondheim \\
Norway} \email{seip@math.ntnu.no}

\thanks{Saksman's research was supported in part by the Lars Onsager Professorship at NTNU and in part by the Finnish Academy CoE ``Analysis and Dynamics''. The research of Bondarenko, Brevig, and Seip was supported in part by Grant 227768 of the Research Council of Norway. }

\title{Linear space properties of $H^p$ spaces of Dirichlet series}

\begin{abstract}
	We study $H^p$ spaces of Dirichlet series, called $\mathcal{H}^p$, for the range $0<p< \infty$. We begin by showing that two natural ways to define $\mathcal{H}^p$ coincide. We then proceed to study some linear space properties of $\mathcal{H}^p$.
More specifically, we study linear functionals generated by fractional primitives of the Riemann zeta function; our estimates rely on certain Hardy--Littlewood inequalities and display an interesting phenomenon, called  contractive symmetry between $\mathcal{H}^p$ and $\mathcal{H}^{4/p}$, contrasting the usual $L^p$ duality. We next deduce general coefficient estimates, based on an interplay between the multiplicative structure of $\mathcal{H}^p$ and certain new one variable bounds. Finally, we deduce general estimates for the norm of the partial sum operator $\sum_{n=1}^\infty a_n n^{-s}\mapsto \sum_{n=1}^N a_n n^{-s}$ on $\Hp$ with $0< p \le 1$, supplementing a classical result of Helson for the range $1<p<\infty$. The results for the coefficient estimates and for the partial sum operator exhibit the traditional schism between the ranges $1\le p \le \infty$ and $0<p<1$.
\end{abstract}
\maketitle

\section{Introduction}

$H^p$ spaces of Dirichlet series, to be called $\mathcal{H}^p$ in what follows, have received considerable attention in recent years but mostly in the Banach space case $1\leq p <\infty$.  In the present paper, we explore $\mathcal{H}^p$ in the full range $0<p<\infty$, which in part can be given a number theoretic motivation: the interplay between the additive and multiplicative structure of the integers is displayed in a more transparent way by the results obtained without any a priori restriction on the exponent $p$. 

The emerging theory of $\mathcal{H}^p$ differs in many aspects  from that of the classical Hardy spaces. Unforeseen phenomena appear, some related to the complicated structure of the dual of $\mathcal{H}^p$ and others arising from number theory. In the present paper, we set out to study some of the most classical questions related to the linear space structure of $\mathcal{H}^p$; this will again lead us to some of the interesting features of $\Hp$ not encountered in the classical setting, and consequently our results shed new light on them.  

By a basic observation of Bohr, the multiplicative structure of the integers allows us to view an ordinary Dirichlet series of the form
\[ f(s)=\sum_{n=1}^\infty a_n n^{-s} \]
as a function of infinitely many variables. Indeed, by the transformation $z_j=p_j^{-s}$ (here $p_j$ is the $j$th prime number) and the fundamental theorem of arithmetic, we have the Bohr correspondence, 
\begin{equation}\label{eq:bohr} 
	f(s):= \sum_{n=1}^\infty a_{n} n^{-s}\quad\longleftrightarrow\quad \mathcal{B}f(z):=\sum_{n=1}^{\infty} a_n z^{\kappa(n)}, 
\end{equation}
where we use multi-index notation and $\kappa(n)=(\kappa_1,\ldots,\kappa_j,0,0,\ldots)$ is the multi-index such that $n = p_1^{\kappa_1} \cdots p_j^{\kappa_j}$. This transformation---the so-called Bohr lift---gives an isometric isomorphism between $\Hp$ and the Hardy space $H^p(\D^\infty)$. 

We start by showing in  Section \ref{se:prel}  that the Bohr lift allows for  a canonical definition of $\Hp$ in the full range $0<p<\infty$, and this definition agrees with the natural one obtained by asking ``der $m$te Abschnitt'' to lie uniformly in  $H^p(\D^m)$. We have chosen to be quite detailed in this groundwork, because the infinite-dimensional situation and the non-convexity of the $L^p$ quasi-norms for $0<p<1$ require some extra care. At the end of the section, we also summarize briefly some known facts and easy consequences, such as for instance how some results for $\Ht$ can be transferred to $\Hp$ when either $p=2k$ or $p=1/(2k)$ for $k=2,3,..$. 

In  Section~\ref{sec:HL} we investigate certain linear functionals generated by fractional primitives of the Riemann zeta function. We want to characterize when they belong to a given $\mathcal{H}^p$ space or the dual.  We will see that in this situation the standard duality relation is replaced what we call contractive symmetry between $\mathcal{H}^p$ and $\mathcal{H}^{4/p}$. We refer to these estimates as multiplicative because they seem to arise in multiplicative or almost multiplicative situations. An example of this was given already in \cite{MS11} where it was observed that if $\varphi$ is a Dirichlet series in the dual of $\mathcal{H}^1$ with multiplicative coefficients, then $\varphi\in \mathcal{H}^4$ and $4$ is the largest possible exponent in general. The contractive symmetry is displayed strongly in the results of this section since the fractional primitives of the Riemann zeta studied here are in some sense almost multiplicative. We note in passing that, surprisingly, there remain basic problems related to the contractive symmetry that are still open in the case of the unit disc (see \cite{BOSZ}).

In Section \ref{se:coeff} we investigate individual coefficient estimates, which are of special interest only in the case $0<p<1$. The estimate for the coefficient $a_n$ (in front of $n^{-s}$) will depend solely on  the multiplicative structure of $n$. We will combine this observation with new one variable bounds in order to exhibit nontrivial estimates for the maximal order in terms of the size of $n$. 

The additive structure of the integers plays a role whenever we restrict attention to the properties of $f(s)$ viewed as an analytic function in a half-plane or when we consider any problem for which the order of summation matters. A particularly interesting example is that of the partial sum operator
\[ S_N f(s):=\sum_{n=1}^N a_n n^{-s}, \]
viewed as an operator on $\Hp$. By a classical theorem of Helson \cite{H1}, we know that it is uniformly bounded on $\Hp$ when $1<p<\infty$. In Section~\ref{se:partial}, we will give bounds that are essentially best possible in the range $0<p<1$ and an improvement by a factor $1/\log\log N$ on the previously known bounds when $p=1$. We are however still far from knowing the precise asymptotics of the norm of $S_N$ when it acts on either $\Ho$ or $\Hi$. 

To close this introduction, we note that there are many questions about $\Hp$ that are not treated or only briefly mentioned in our paper.  For further information about known results and open problems, we refer to the monograph \cite{QQ13} and the recent paper \cite{SS17}. Finally, we note that  a closely related paper \cite{BBSSZ} addresses a number theoretic problem that deals with the \emph{interplay} between the additive and multiplicative structure of the integers, namely the computation of what are known as the pseudomoments of the Riemann zeta function.

\subsection*{Notation} We will use the notation $f(x) \lesssim g(x)$ if there is some constant $C>0$ such that $|f(x)|\leq C|g(x)|$ for all (appropriate) $x$. If we have both $f(x) \lesssim g(x)$ and $g(x) \lesssim f(x)$, we will write $f(x)\simeq g(x)$. As above, the increasing sequence of prime numbers will be denoted by $(p_j)_{j\geq1}$, and the subscript will sometimes be dropped when there can be no confusion. The number of prime factors in $n$ will be denoted by $\Omega(n)$ (counting multiplicities). We will also use the standard notations $\lfloor x \rfloor=\max\{n\in \mathbb{N}: n\le x \}$ and  $\lceil x \rceil=\min\{n\in \mathbb{N}: n\ge x \}$.


\section{Definitions and basic properties of $\Hp$ and $H^p(\D^\infty)$}\label{se:prel}

\subsection{Definition of $H^p(\D^\infty)$} We use the standard notation $\T:=\{ z\,:\,|z|=1\}$ for the unit circle which is the boundary of the unit disc $\D:=\{z\,:\, |z|<1\}$ in the complex plane, and we equip $\mathbb{T}$ with normalized one-dimensional Lebesgue measure $\mu$ so that $\mu(\T)=1.$ We write $\mu_d:=\mu\times \cdots \times \mu$ for the product of $d$ copies of $\mu$, where $d$ may belong to $\N\cup\{\infty\}.$

We begin by recalling that for every $p>0$, the classical Hardy space $H^p(\D)$ (also denoted by $H^p(\T)$) consists of analytic functions $f:\D\to\C$ such that
\[ \|f\|_{H^p(\D)}^p:=\sup_{0<r<1}\int_\T |f(rz)|^p\,d\mu(z) <\infty. \]
This is a Banach space (quasi-Banach in case $0<p<1$), and polynomials are dense in $H^p(\D)$, so it could as well be defined as the closure of all polynomials in the above norm (or quasi-norm). We refer to \cite{Duren} or the first chapters of \cite{G} for the definition and basic properties of the Hardy spaces on $\D$.

For the finite dimensional polydisc $\D^d$ with $d\geq 2$, the definition of Hardy spaces can be made in a similar manner: For every $p>0$, a function $f:\D^d\to\C$ belongs to $H^p(\D^d)$ when it is analytic separately with respect to each of the variables $z_1,\ldots, z_d$ and
\[ \|f\|_{H^p(\D^d)}^p:=\sup_{r<1}\int_{\T^d} |f(rz)|^p\,d\mu_d(z) <\infty. \]
The standard source for these spaces is Rudin's monograph \cite{Ru}. As in the one-dimensional case, for almost every $z$ in $\mathbb{T}^d$, the radial boundary limit
\[f^*(z):=\lim_{r\to 1^-} f(rz)\]
exists, and we may write 
\begin{equation}\label{eq:01} 
	\| f\|_{H^p(\D^d)}^p=\int_{\T^d} |f^*(z)|^p\,d\mu_d(z). 
\end{equation}
This means that $H^p(\mathbb{D}^d)$ is a subspace of $L^p(\mathbb{T}^d,\mu_d)$. Moreover, again as in the one-dimensional case, for every $f$ in $H^p(\D^d)$, we have that 
\begin{equation}\label{eq:02} 
	\lim_{r\to 1^{-1}} \|f-f_r\|_{H^p(\D^d)}=0, 
\end{equation}
where $f_r(z):=f(rz)$. This implies that the polynomials are dense in $H^p(\D^d)$, so that the space could equally well be defined as the closure of all polynomials with respect to the norm on the boundary given by \eqref{eq:01}.

Both \eqref{eq:01} and \eqref{eq:02} are most conveniently obtained by applying the $L^p$-boundedness of the radial maximal function on $H^p(\D^d)$ for all $p>0$, a result which can be obtained by considering a dummy variable $w$ in $\D$ and checking first that, given $f$ in $H^p(\D^d)$, the function
\[ w\mapsto f(wz_1,\ldots , wz_d) \]
lies in $H^p(\D^d)$ for almost every $(z_1,\ldots z_d)\in \T^d$. By Fubini's theorem, the boundedness of the maximal function then reduces to the classical one-dimensional estimate.

In order to define $H^p(\D^\infty)$, some extra care is needed because functions in $H^p(\D^\infty)$ will in general not be well defined in the whole set $\D^\infty$. To keep things simple, we henceforth consider the set $\D^\infty_\fin$ which consists of elements $z=(z_j)_{j\geq 1}\in\D^\infty$ such that $z_j\neq 0$ only for finitely many $k$. A function $f:\D^\infty_\fin\to\C$ is analytic if it is analytic at every point $z$ in $\D^\infty_\fin$ separately with respect to each variable. Obviously any analytic $f:\D^\infty_\fin\to\C$ can be written by a convergent Taylor series
\[f(z)=\sum_{\kappa\in \N^\infty_\fin}c_\kappa z^\kappa,\quad z\in \D^\infty_\fin,\]
and the coefficients $c_\kappa$ determine $f$ uniquely. The truncation $A_mf$ of $f$ onto the first $m$ variables $A_mf$ (called ``der $m$te Abschnitt'' by Bohr) is defined as
\[A_mf(z_1,z_2,\ldots)=f(z_1,\ldots , z_m,0,0,\ldots)\]
for every $z$ in $\D^\infty_\fin$. By applying the fundamental estimate $|g(0)|\leq \|g\|_{H^p(\D^d)}$, obtained by iterating the case $d=1$, we deduce that 
\begin{equation}\label{eq:03} 
	\| A_mf\|_{H^p(\D^m)}\leq \| A_{m'}f\|_{H^p(\D^{m'})} 
\end{equation}
whenever $m'\geq m$. 
\begin{definition}
	Let $p>0$. The space $H^p(\D^\infty)$ is the space of analytic functions on $\D^\infty_\fin$ obtained by taking the closure of all polynomials in the norm (quasi-norm for \ $0<p<1$)
	\[\|f\|_{H^p(\D^\infty)}^p :=\int_{\T^\infty}|f(z)|^p\,d\mu_\infty (z).\]
\end{definition}
Fix a compact set $K$ in $\D^d$ and embed it as the subset $\widetilde K$ of $ \D^\infty$ so that
\[ \widetilde K:= \left\{ z=(z_1,\ldots, z_d, 0,0, \ldots)\in \D^\infty\,:\, \ (z_1,\ldots,z_d)\subset K\right\}. \]
For all polynomials $g$ we clearly have $\sup_{z\in\widetilde K}|g(z)|\leq C_K\|g\|_{H^p(\D^\infty)}$. It follows that any limit of polynomials is analytic on $\D^\infty_\fin$, whence $H^p(\D^\infty)$ is well defined. This also implies that every element $f$ in $H^p(\D^\infty)$ has a well-defined Taylor series $f(z) = \sum_{\kappa} c_\kappa z^\kappa$ and, in turn, this Taylor series determines $f$ uniquely. Namely, by recalling \eqref{eq:03}, we have that $A_mf$ is in $H^p(\D^m)$ for every $m\geq 1$ and the $A_m f$ are certainly determined by the Taylor series. Finally, by polynomial approximation, it follows that
\[\lim_{m\to\infty} \| f- A_mf\|_{H^p(\D^\infty)}=0.\]
Obviously, if a function $f$ in $H^p(\D^\infty)$ depends only on the variables $z_1,\ldots z_d$, then we have $\|f\|_{H^p(\D^\infty)}=\|f\|_{H^p(\D^d)}$.

Cole and Gamelin \cite{CG} established an optimal estimate for point evaluations on $H^p(\D^\infty)$ by showing that 
\begin{equation}\label{eq:CGestimate} 
	|f(z)|\leq \left(\prod_{j=1}^\infty \frac{1}{1-|z_j|^2}\right)^{1/p}\|f\|_{H^p(\mathbb{D}^\infty)}. 
\end{equation}
Thus the elements in the Hardy spaces continue analytically to the set $\D^\infty \cap \ell^2$.

If $f$ is an integrable function (or a Borel measure) on $\T^\infty$, then we denote its Fourier coefficients by
\[\widehat f(\kappa):=\int_{\T^\infty}f(z) \overline{z^\kappa}\, d\mu_\infty(z)\]
for multi-indices $\kappa$ in $\mathbb{Z}^\infty_{\fin}$. When $p\geq 1$, it follows directly from the definition of $H^p(\D^\infty)$ that it can be identified as the analytic subspace of $L^p(\T^\infty)$, consisting of the elements in $L^p(\T^\infty)$ whose non-zero Fourier coefficients lie in the positive cone $\N^\infty_{\fin}$ (called the ``narrow cone'' by Helson \cite{Helson06}).

The following result verifies that, alternatively, $H^p(\D^\infty)$ may be defined in terms of the uniform boundedness of the $L^p$-norm of the sequence $A_m f$ for $m\geq 1$, and the functions $A_m f$ approximate $f$ in the norm of $H^p(\D^\infty)$. 
\begin{theorem}\label{th:Abschnitte} 
	Suppose that $0<p<\infty$ and that $f$ is a formal infinite dimensional Taylor series. Then $f$ is in $H^p(\D^\infty)$ if and only if 
	\begin{equation}\label{eq:mart} 
		\sup_{m\geq 1}\| A_mf\|_{H^p(\D^m)}<\infty. 
	\end{equation}
	Moreover, for every $f$ in $H^p(\D^\infty)$, it holds that $\|A_mf- f\|_{H^p(\D^\infty)}\to 0$ as $m\to\infty.$ 
\end{theorem}
\begin{proof}
	[Proof for the case $p\ge 1$] When $p>1$, the statements follow from the fact that $(A_mf)_{m\geq 1}$ is obviously an $L^p$-martingale sequence with respect to the natural sigma-algebras. It follows in particular that there is an $L^p$ limit function (still denoted by $f$) of the sequence $A_mf$ on the distinguished boundary $\T^\infty$, which has the right Fourier series, and the density of polynomials follows immediately from the finite-dimensional approximation. In the case $p=1$, this fact is stated in \cite[Cor.~3]{AOS2}, and is derived as consequence of the infinite-dimensional version of the brothers Riesz theorem on the absolute continuity of analytic measures, due to Helson and Lowdenslager \cite{HL} (a simpler proof of the result from \cite{HL} is also contained in \cite{AOS2}). The approximation property of the $A_mf$ then follows easily. 
\end{proof}
The case $0<p<1$ requires a new argument and will be presented in the next subsection.

\subsection{Proof of Theorem \ref{th:Abschnitte} for $0<p<1$}\label{se:two} Our aim is to prove Lemma~\ref{pr:1} below, from which the claim will follow easily. In an effort to make the computations of this section more readable, we temporarily adopt the convention that $\|f\|_{L^p(\mathbb{T}^d)} = \|f\|_p$, where it should be clear from the context what $d$ is. We start with the following basic estimate. 
\begin{lemma}\label{le:4} 
	Let $0<p<1$. There is a constant $C_p<\infty$ such that all (analytic) polynomials $f$ on $\mathbb{T}$ satisfy the inequality 
	\begin{equation}\label{basicp} 
		\| f-f(0)\|_p^p \le C_p\left( \| f\|_p^p - |f(0)|^p + |f(0)|^{p-p^2/2}\left(\| f\|_p^p - |f(0)|^p\right)^{p/2} \right). 
	\end{equation}
\end{lemma}
\begin{proof}
	In this proof, we use repeatedly the elementary inequality $ |a+b|^p\le |a|^p+|b|^p$, which is our replacement for the triangle inequality. We see in particular, by this inequality and the presence of the term $\| f\|_p^p - |f(0)|^p$ inside the brackets on the right-hand side, that \eqref{basicp} is trivial if, say, $\| f \|_p^p\ge (3/2) |f(0)|^2$. We may therefore disregard this case and assume that $f$ satisfies $f(0)=1$ and $\|f\|^p_p=1+\varepsilon $ with $\varepsilon <1/2.$ Our aim is to show that, under this assumption, 
	\begin{equation}\label{keyprove} 
		\| f-1\|_p^p \le C_p \varepsilon^{p/2}. 
	\end{equation}
	
	We begin by writing $f=UI$, where $U$ is an outer function and $I$ is an inner function, such that $U(0)>0$. By subharmonicity of $|U|^p$, we have $1\leq |U(0)|\le (1+\varepsilon)^{1/p}\le 1+c_p\varepsilon.$ This means that $I(0)\ge (1+c_p\varepsilon)^{-1}\ge 1-c_p\varepsilon$. We write $f-1=(U-1)I+I-1$ and obtain consequently that 
	\begin{equation}\label{summands} 
		\| f-1\|_p^p\le \| U-1\|_p^p+\| I-1\|_p^p. 
	\end{equation}
	In order to prove \eqref{keyprove}, it is therefore enough to show that each of the two summands on the right-hand side of \eqref{summands} is bounded by a constant times $\varepsilon^{p/2}$.
	
	We begin with the second summand on the right-hand side of \eqref{summands} for which we claim that 
	\begin{equation}\label{Iint} 
		\| I-1\|_p^p\le C_p' \varepsilon^{p/2} 
	\end{equation}
	holds for some constant $C_p'$. We write $I=u+iv$, where $u$ and $v$ are respectively the real and imaginary part of $I$. Since $1-u\ge 0$, we see that 
	\begin{equation}\label{eq:1} 
		\| 1-u \|_1=\int_{\T} (1-u(z) )dm(z) = 1-I(0) \le c_p \varepsilon. 
	\end{equation}
	Using H\"{o}lder's inequality, we therefore find that 
	\begin{equation}\label{eq:2} 
		\| 1-u \|_p^p \le c_p^P \varepsilon^p. 
	\end{equation}
	In view of \eqref{eq:1} and using that $|I|=1$ and $(1-u^2)\leq 2(1-u)$, we also get that
	\[ \|v\|_p^p \le \| v\|_2^{p}=\| 1-u^2\|_1^{p/2} \le \left(2 \| 1-u\|_1\right)^{p/2} \le (2c_p)^{p/2} \varepsilon^{p/2}. \]
	Combining this inequality with \eqref{eq:2}, we get the desired bound \eqref{Iint}.
	
	We turn next to the first summand on the right-hand side of \eqref{summands} and the claim that 
	\begin{equation}\label{Uint} 
		\|U-1\|_p^p\le C_p'' \varepsilon^{p/2} 
	\end{equation}
	holds for some constant $C_p''$. By orthogonality, we find that
	\[ \|U^{p/2}-U(0)^{p/2}\|_2^2\le \varepsilon\]
	and hence 
	\begin{equation}\label{basic} 
		\| U^{p/2}-1 \|_2 \le \|U^{p/2}-U(0)^{p/2}\|_2+ (U(0)^{p/2}-1)^{1/2} \le 2 \varepsilon^{1/2}. 
	\end{equation}
	Since $|U^{p/2}-1|\ge ||U|^{p/2}-1|\ge (p/2)\log_+ |U|$ and $U(0)\geq 1$, this implies that 
	\begin{equation}\label{second} 
		\|\log |U|\|_1=2 \| \log_+|U|\|_1 -\log |U(0)| \le 8p^{-1} \varepsilon^{1/2}. 
	\end{equation}
	It follows that
	\[ m\left(\left\{z:\ |\log|U(z)||\ge \lambda\right\}\right)\le 8(p \lambda )^{-1} \varepsilon^{1/2} \quad \text{and} \quad m\left(\left\{z:\ |\arg U(z)|\ge \lambda \right\}\right)\le C\lambda^{-1} \varepsilon^{1/2}, \]
	where the latter inequality is the classical weak-type $L^1$ estimate for the conjugation operator. We now split $\T$ into three sets 
	\begin{align*}
		E_1&:=\left\{z\,: \, |U(z)|>3/2\} \cup \{ z:\ |U(z)|<1/2 \right\}, \\
		E_2 &:= \left\{z\,: \, 1/2\le |U(z)|\le 3/2, \, |\arg U(z)|\ge \pi/4 \right\}, \\
		E_3 &:= \T\setminus(E_1\cup E_2). 
	\end{align*}
	It is immediate from \eqref{basic} that
	\[ \| \chi_{E_1} (U-1)\|^p_p \lesssim \varepsilon. \]
	Since $m(E_2)\le C \varepsilon^{1/2}$, we have trivially that
	\[ \| \chi_{E_2} (U-1)\|^p_p \le C (5/2)^p \varepsilon^{1/2}. \]
	Finally, on $E_3$, we have that $|U^{p/2}-1|\simeq |U-1|$, and so it follows from \eqref{basic} and H\"older's inequality that
	\[ \| \chi_{E_3} (U-1)\|^p_p \lesssim \varepsilon^{p/2}. \]
	Now the desired inequality \eqref{Uint} follows by combining the latter three estimates. 
\end{proof}

One may notice that that in the last step of the proof above we could have used \eqref{second} and the fact that the conjugation operator is bounded from $L^1$ to $L^p$. It seems that the exponent $p/2$ is the best we can get. It is also curious to note that with $p=2/k$ and $k\geq 2$ an integer, one could avoid the use of the weak-type estimate for $\arg U$ and get a very slick argument by simply observing that if $g=U^{p/2}$ and $\omega_1,\ldots ,\omega_k$ are the $k$th roots of unity, then by H\"older's inequality,
\[ \|U-1\|_p\leq \prod_{j=1}^k\|g-\omega_j\|_2, \]
and on the right hand side one $L^2$-norm is estimated by $\varepsilon^{1/2}$ and the others by a constant since we are assuming $\varepsilon \leq 1/2.$ Again one could raise the question if one can interpolate to get all exponents. 
\begin{lemma}\label{pr:1} 
	Suppose that $0<p<1$. If $g$ is a polynomial on $\T^\infty$, then
	\[ \| A_{m+k}g-A_mg\|_p^p\, \le \, C_p \left(\| A_{m+k}g\|_p^p- \| A_{m}g\|_p^p \, + \, \| A_{m}g\|_p^{p -p^2/2}\left( \| A_{m+k}g\|_p^p- \| A_{m}g\|_p^p\right)^{p/2}\right) \]
	holds for arbitrary positive integers $m$ and $k$, where $C_p$ is as in Lemma~\ref{basicp}. 
\end{lemma}
\begin{proof}
	We set $h:=A_{m+k}g$ and view $h$ as a function on $\T^{m}\times \T^k$ so that $A_m g (w,w')=h(w,0)$. Now fix arbitrary points $w$ in $\T^m$ and $w'$ in $\T^k.$ We apply the preceding lemma to the function
	\[ f(z):=h(w,zw'), \]
	which is an analytic function on $\D$. This yields 
	\begin{align*}
		\int_{\T} |h(w,zw')-h(w,0)|^p \,d\mu(z) \le C_p \Bigg(&\int_{\T}|h(w,zw')|^p\,d\mu(z) - |h(w,0)|^p \\
		& + |h(w,0)|^{p-p^2/2}\left(\int_{\T}|h(w,z w')|^p\,d\mu(z) \, -\, |h(w,0)|^p\right)^{p/2}\Bigg). 
	\end{align*}
	The claim follows by integrating both sides with respect to $(w,w')$ over $\T^{m+k}$ and applying H\"older's inequality to the last term on the right-hand side. 
\end{proof}
\begin{proof}
	[Proof of Theorem \ref{th:Abschnitte} for $0<p<1$] If $f$ is in $H^p(\D^\infty)$, then clearly \eqref{eq:mart} holds. To prove the reverse implication, we start from a formal Taylor series $f$ for which \eqref{eq:mart} holds. Then by assumption $A_m f$ is in $H^p(\D^\infty)$, and we have that $A_m(A_{m'}f)=A_mf$ whenever $m'\geq m\geq 1$. Therefore the quasi-norms $\|A_m f\|_{H^p(\D^\infty)}$ constitute an increasing sequence, and hence \eqref{eq:mart} implies that
	\[ \lim_{m\to\infty}\sup_{k\geq 1}\left(\|A_{m+k}f\|_{H^p(\D^\infty)}-\|A_{m}f\|_{H^p(\D^\infty)}\right)=0. \]
	By Lemma~\ref{pr:1}, we therefore find that $(A_mf)_{m\geq 1}$ is a Cauchy sequence in $H^p(\D^\infty)$, whence $f=\lim_{m\to\infty}A_mf$ in $H^p(\D^\infty)$ since an element in $H^p(\D^\infty)$ is uniquely determined by the sequence $A_m f$. 
\end{proof}

\subsection{Definition of $\mathcal{H}^p$} 
The systematic study of the Hilbert space $\Ht$ began with the paper \cite{HLS97} which defined $\Ht$ to be the collection of Dirichlet series
\[ f(s)=\sum_{n=1}^\infty a_n n^{-s}, \]
subject to the condition $\| f\|_\Ht:= \left(\sum_{n=1}^\infty|a_n|^2\right)^{1/2}<\infty$. The space $\Ht$ consists of functions analytic in the half-plane $\mathbb{C}_{1/2}:=\{s=\sigma+it: \ \sigma>1/2\}$,
since the Cauchy--Schwarz inequality shows that the above Dirichlet series converges absolutely for those values of $s$. Bayart \cite{Bayart02} extended the definition to every $p>0$ by defining $\Hp$ as the closure of all Dirichlet polynomials $f(s):=\sum_{n=1}^{N}a_n n^{-s}$ under the norm (or quasi-norm when $0<p<1$) 
\begin{equation}\label{eq:hpdef} 
	\|f\|_{\Hp}:= \left(\lim_{T\to\infty}\frac{1}{2T}\int_{-T}^T|f(it)|^p \,dt\right)^{1/p}. 
\end{equation}
Computing the limit when $p=2$, we see that \eqref{eq:hpdef} gives back the original definition of $\Ht$. However, at first sight it is not clear that the above definition of $\Hp$ is the right one or that it even yields spaces of convergent Dirichlet series in any right half-plane.

The clarification of these matters is provided by the Bohr lift \eqref{eq:bohr}. By Birkhoff's ergodic theorem (or by an elementary argument found in \cite[Sec.~3]{SS09}), we obtain the identity 
\begin{equation}\label{eq:norm} 
	\|f\|_{\Hp}=\|\mathcal{B} f\|_{H^p(\D^\infty)}:=\left(\int_{\T^\infty}|\mathcal{B}f(z)|^p\,d \mu_\infty(z)\right)^{1/p}. 
\end{equation}
Since the Hardy spaces on the infinite dimensional torus $H^p(\D^\infty)$ may be defined as the closure of analytic polynomials in the $L^p$-norm on $\T^\infty$, it follows that the Bohr correspondence gives an isomorphism between the spaces $H^p(\D^\infty)$ and $\Hp$. This linear isomorphism is both isometric and multiplicative, and this results in a fruitful interplay: Many questions in the theory of the spaces $\Hp$ can be better treated by considering the isomorphic space $H^p(\D^\infty)$, and vice versa. An important example is the Cole-Gamelin estimate \eqref{eq:CGestimate} which immediately implies that for every $p>0$ the space $\Hp$ consists of analytic functions in the half-plane $\mathbb{C}_{1/2}$. In fact, we infer from \eqref{eq:CGestimate} that
\[ |f(\sigma+it)|^p\le \zeta(2\sigma) \| f\|_{\Hp}^p \]
holds whenever $\sigma>1/2$, where $\zeta(s)$ is the Riemann zeta function. Moreover, since the coefficients of a convergent Dirichlet series are unique, functions in $\Hp$ are completely determined by their restrictions to the half-plane $\mathbb{C}_{1/2}$. This means in particular that $\Hp$ can be thought of as a space of analytic functions in this half-plane.

To complete the picture, we mention that $\Hi$ is defined as the space of Dirichlet series $f(s)=\sum_{n=1}^\infty a_n n^{-s}$ that represent bounded analytic functions in the half-plane $\sigma>0$. We endow $\Hi$ with the norm
\[ \| f\|_{\Hi}:=\sup_{\sigma>0} |f(s)|, \quad s=\sigma+it, \]
and then the Bohr lift allows us to associate $\Hi$ with $H^\infty(\D^\infty)$. We refer to \cite{QQ13} for this fact and further details about the interesting and rich function theory of $\Hi$.


\subsection{Summary of known results} The function theory of the two distinguished spaces $\Ht$ and $\Hi$ is by now quite well developed; we refer again to \cite{QQ13, SS17} for details. The results for the range $1\le p < \infty$, $p\neq 2$, are less complete. In this section, we mention briefly some key results that extend to the whole range $0<p<\infty$, as well as some familiar difficulties that arise in our attempts to make such extensions.

We begin with the theorem on multipliers that was first established in \cite{HLS97} for $p=2$ and extended to the range $1\le p < \infty$ in \cite{Bayart02}. We recall that a multiplier $m$ for $\Hp$ is a function such that the operator $f\mapsto mf$ is bounded on $\Hp$, and the multiplier norm is the norm of this operator. The theorem on multipliers asserts that the space of multipliers for $\Hp$ is equal to $\Hi$, and this remains true for $0<p<1$, by exactly the same proof as in \cite{Bayart02}. Another result that carries over without any change, is the Littlewood--Paley formula of \cite[Sec.~5]{BQS}. The latter result was already used in \cite{BPS16}.

For some results, only a partial extension from the case $p=2$ is known to hold. A well known example is whether the $L^p$ integral of a Dirichlet polynomial $f(s)=\sum_{n=1}^N a_n n^{-s}$ over any segment of fixed length on the vertical line $\mre s=1/2$ is bounded by a universal constant times $\| f \|_{\Hp}^{p}$. This is known to hold for $p=2$ and thus trivially for $p=2k$ for $k$ a positive integer. As shown in \cite{OS12}, this embedding holds if and only if the following is true: The boundedly supported Carleson measures for $\Hp$ satisfy the classical Carleson condition in $\C_{1/2}$.

There is an interesting counterpart for $p<2$ to the trivial embedding for $p=2k$ and $k$ a positive integer $>1$. This is the following statement about interpolating sequences. If $S=(s_j)$ is a bounded interpolating sequence in $\C_{1/2}$, then we can solve the interpolation problem $f(s_j)=a_j$ in $\Hp$ when
\[ \sum_j |a_j|^p(2\sigma_j-1)<\infty \]
and $p=2/k$ for $k$ a positive integer. Indeed, choose any $k$th root $a_j^{1/k}$ and solve $g(s_j)=a_j^{1/k}$ in $\Ht$. Then $f=g^k$ solves our problem in $\Hp$. We do not know if this result extends to any $p$ which is not of the form $p=2/k$. Comparing the two trivial cases, we observe that there is an interesting ``symmetry'' between the embedding problem for $\Hp$ and the interpolation problem for $\mathcal{H}^{4/p}$. A similar phenomenon was observed in \cite{BOSZ} and will also be explored in the next section. 

\section{Linear functionals associated with the Riemann zeta function} \label{sec:HL} 
It was asked in \cite[Sec.~5]{Hilbert16} whether the primitive of the half-shift of the Riemann zeta function
\[\varphi(s): = 1 + \sum_{n=2}^\infty \frac{1}{\sqrt{n}\log{n}} n^{-s}\]
defines a bounded linear functional on $\mathcal{H}^1$, or equivalently: Is there a constant $C$ such that 
\begin{equation}\label{eq:hilbert} 
	\left|a_1 + \sum_{n=2}^N \frac{a_n}{\sqrt{n}\log{n}}\right| \leq C \|f\|_{\mathcal{H}^p} 
\end{equation}
for every Dirichlet polynomial $f(s) = \sum_{n=1}^N a_n n^{-s}$ when $p=1$? Clearly, \eqref{eq:hilbert} is satisfied if $p=2$, and it was shown in \cite{BB16} that \eqref{eq:hilbert} holds whenever $p>1$. It was also demonstrated in \cite{BB16} that $\varphi$ is in $\mathcal{H}^p$ if and only if $p<4$. 

We are still not able to answer the original question from \cite{Hilbert16}, but we will prove some complementary results that shed more light on this and related questions about duality. For $\beta>0$, consider the following fractional primitives of the half-shift of the Riemann zeta function:
\begin{equation}\label{eq:betafunc}\varphi_\beta(s) := 
1 + \sum_{n=2}^\infty \frac{1}{\sqrt{n}(\log{n})^\beta}\,n^{-s}. \end{equation}
We are interested in the following questions. 
\begin{enumerate}
	\item[(a)] For which $\beta>0$ is $\varphi_\beta$ in $\mathcal{H}^p$, when $2 \leq p < \infty$? 
	\item[(b)] For which $\beta>0$ is $\varphi_\beta$ in $(\mathcal{H}^p)^\ast$, when $0<p\leq 2$? 
\end{enumerate}
Before proceeding, let us clarify question (b). The linear functional generated by $\varphi_\beta$ can be expressed as
\[\langle f, \varphi_\beta \rangle_{\mathcal{H}^2} := a_1 + \sum_{n=2}^\infty \frac{a_n}{\sqrt{n}(\log{n})^\beta},\]
when $f(s)=\sum_{n=1}^\infty a_n n^{-s}$. We say that the linear functional generated by $\varphi_\beta$ acts boundedly on $\mathcal{H}^p$, or equivalently that $\varphi_\beta$ is in $(\mathcal{H}^p)^\ast$, if there is a constant $C>0$ such that
\[|\langle f, \varphi_\beta \rangle_{\mathcal{H}^2}| \leq C \|f\|_{\mathcal{H}^p}\]
for every Dirichlet polynomial $f(s) = \sum_{n=1}^N a_n n^{-s}$. Our result is: 
\begin{theorem}\label{thm:duality} 
	Suppose that $\beta>0$.
	\begin{enumerate}
		\item[(a)] Let $2\leq p < \infty$. Then $\varphi_\beta$ is in $\mathcal{H}^p$ if and only if $\beta>p/4$. 
		\item[(b)] Let $0<p\leq 2$. If $\beta>1/p$ then $\varphi_\beta$ is in $(\mathcal{H}^p)^\ast$ and if $\beta<1/p$ then $\varphi_\beta$ is not in $(\mathcal{H}^p)^\ast$. 
	\end{enumerate}
\end{theorem}

It is well-known that the dual space $(\mathcal{H}^p)^\ast$ for $1<p<\infty$ is not equal to $\mathcal{H}^q$ with $p^{-1}+q^{-1}=1$ (see \cite[Sec.~3]{SS09}). Theorem~\ref{thm:duality} provides additional examples illustrating this fact. One may observe in Theorem \ref{thm:duality} the first appearance in the present paper of the contractive symmetry between (the dual of)  $\Hp$ and $\mathcal{H}^{4/p}$. Namely,  the result (loosely speaking) tells us that if $\varphi_\beta$ in $\mathcal{H}^{4/p}$ then $\varphi_\beta$ is in $(\mathcal{H}^{p})^\ast$.

The proof relies on Hardy--Littlewood inequalities for Dirichlet series. The general divisor function $d_\alpha(n)$ for $\alpha\geq1$ is defined by the rule 
\begin{equation}\label{eq:dadef} 
	\zeta^\alpha(s) = \sum_{n=1}^\infty d_\alpha(n) n^{-s}, \quad \sigma>1.
\end{equation}
By the Euler product of $\zeta(s)$ and the binomial series, we note that $d_\alpha(n)$ is a multiplicative function whose value at the prime powers is
\[d_\alpha(p_j^k) = \binom{k+\alpha-1}{k}.\]
If $k$ is an integer, then $d_k(n)$ denotes the number of ways we may write $n$ as a product of $k$ positive integers.  In this case, we have (see \cite[Lem.~3]{BBSSZ}) 
\begin{align*}
	\left(\sum_{n=1}^\infty \frac{|a_n|^2}{d_k(n)}\right)^\frac{1}{2} &\leq \|f\|_{2/k}, \\
	\|f\|_{2k} &\leq \left(\sum_{n=1}^\infty |a_n|^2 d_k(n)\right)^\frac{1}{2}.
\end{align*}
if $f(s) = \sum_{n=1}^\infty a_n n^{-s}$. Moreover, it is conjectured (see \cite{BOSZ}) that these inequalities in fact hold for all real numbers $k\geq1$. As a replacement, we will use the following weaker result from \cite{BBSSZ}, which is obtained by a kind of completely multiplicative interpolation between the integers $k\geq1$ in the inequalities above.

\begin{lemma} \label{lem:multmixedineqs}
	For $\alpha\geq1$, let
	\begin{equation}\label{eq:multmixweight} 
		\Phi_\alpha(n) := d_{\lfloor \alpha \rfloor}(n)\, \left(\frac{\alpha}{\lfloor \alpha \rfloor}\right)^{\Omega(n)}.
	\end{equation}
	If $f(s) = \sum_{n=1}^\infty a_n n^{-s}$, then 
	\begin{align}
		\left(\sum_{n=1}^\infty \frac{|a_n|^2}{\Phi_{2/p}(n)}\right)^\frac{1}{2} &\leq \|f\|_p, & p\leq 2, \label{eq:multmix2} \\
		\|f\|_p &\leq \left(\sum_{n=1}^\infty |a_n|^2 \Phi_{p/2}(n)\right)^\frac{1}{2}, & p\geq2. \label{eq:multmix1} 
	\end{align}
\end{lemma}
Clearly, if $\alpha$ is an integer, then $\Phi_\alpha(n) =d_\alpha(n)$. Let $\mu(n)$ denote the M\"obius function, which is $1$ if $n=1$, $(-1)^{\Omega(n)}$ if $n$ is square-free, and $0$ otherwise. Since
\[|\mu(n)| \Phi_\alpha(n) = |\mu(n)| d_\alpha(n) = |\mu(n)| \alpha^{\Omega(n)},\]
the average order of $\Phi_\alpha$ and $d_\alpha$ is (up to a constant) the same (see \cite[Lem.~7]{BBSSZ}).
Hence we find that
\begin{align}
	\sum_{n\leq x} \frac{d_\alpha(n)}{n} &= \frac{1}{\Gamma(\alpha+1)} (\log{x})^\alpha + O\left((\log{x})^{\alpha-1}\right), \label{eq:davg} \\
	\sum_{n\leq x} \frac{\Phi_\alpha(n)}{n} &= C_\alpha (\log{x})^{\alpha}+ O\left((\log{x})^{\alpha-1}\right). \label{eq:avgorder}
\end{align}
We are now ready to proceed with the proof of Theorem~\ref{thm:duality}.

\begin{proof}
	[Proof of Theorem~\ref{thm:duality} (a)] To begin with, we notice that \eqref{eq:multmix1} implies that
	\[\|\varphi_\beta\|_{\mathcal{H}^p}^2 \leq 1 + \sum_{n=2}^\infty \frac{\Phi_{p/2}(n)}{n(\log{n})^{2\beta}}.\]
	The series on the right-hand side is convergent when $2\beta>p/2$, by \eqref{eq:avgorder} and Abel summation, and we have thus proved that $\varphi_\beta$ is in $\Hp$ whenever $\beta>p/4$.
	
	To settle the case $\beta = p/4$, we set $k=\lfloor p\rfloor$, $q=p/k$, and
	\[ \log^* n=
	\begin{cases}
		\log n, & n>1 \\
		1, & n=1.
	\end{cases}
	\]
	 We only consider square-free integers in \eqref{eq:multmix2} to the effect that 
	\begin{align*}
		\|\varphi_\beta\|_{\mathcal{H}^p}^p = \|\varphi_\beta^k\|_{\mathcal{H}^q}^q &\geq \left(\sum_{n=1}^\infty \frac{|\mu(n)|}{d_{2/q}(n)}\,\frac{1}{n}\left|\sum_{n_1\cdots n_k = n} \frac{1}{(\log^*{n_1})^{\beta}\cdots (\log^*{n_k})^\beta}\right|^2\right)^\frac{q}{2} \\
		&\geq \left(\sum_{n=2}^\infty \frac{|\mu(n)|}{d_{2/q}(n)}\frac{[d_k(n)]^2}{n(\log{n})^{2k\beta}}\right)^\frac{q}{2} = \left(\sum_{n=2}^\infty \frac{|\mu(n)|d_{p[p]/2}(n)}{n(\log{n})^{p[p]/2}}\right)^\frac{q}{2}, 
	\end{align*}
	where we used thrice that $|\mu(n)|\Phi_\alpha(n) = |\mu(n)|d_\alpha(n) = |\mu(n)|\alpha^{\Omega(n)}$. To see that the final series is divergent, we use Abel summation and the estimate 
	\[\sum_{n\leq x} \frac{|\mu(n)|d_\alpha(n)}{n} = D_\alpha (\log{x})^{\alpha-1} + O\left((\log{x})^{\alpha-2}\right),\]
	which is \eqref{eq:avgorder} for squarefree numbers.
\end{proof}
\begin{proof}
	[Proof of Theorem~\ref{thm:duality} (b)] The first statement follows from \eqref{eq:multmix2}, since the Cauchy--Schwarz inequality gives that
	\[|\langle f, \varphi_\beta\rangle_{\mathcal{H}^2}| \leq \left(\sum_{n=1}^\infty \frac{|a_n|^2}{\Phi_{2/p}(n)}\right)^\frac{1}{2}\left(1 +\sum_{n=2}^\infty \frac{\Phi_{2/p}(n)}{n(\log{n})^{2\beta}}\right)^\frac{1}{2}.\]
	Abel summation again gives that the final sum is convergent if $2\beta>2/p$. For the second part, suppose that $\beta<1/p$ and set
	\[f(s) = \left(\prod_{p_j\leq N} \frac{1}{1-p_j^{-1/2-s}}\right)^{2/p}.\]
	Clearly, $\|f\|_{\mathcal{H}^p} \simeq (\log{N})^{1/p}$. We use Abel summation and \eqref{eq:davg} and find that
	\[\langle f, \varphi_\beta \rangle_{\mathcal{H}^2} \geq \sum_{n=2}^N \frac{d_{2/p}(n)}{n(\log{n})^\beta} \simeq (\log{N})^{2/p-\beta}.\]
	We conclude that
	\[\frac{\langle f, \varphi_\beta \rangle_{\mathcal{H}^2}}{\|f\|_{\mathcal{H}^p}} \simeq (\log{N})^{1/p-\beta}\]
	is unbounded as $N\to\infty$, since by assumption $\beta<1/p$. 
\end{proof}

The proof of Theorem~\ref{thm:duality} (b) does not provide any insight into the critical exponent $\beta=1/p$, except for the trivial case $p=2$. Let us collect some observations on this interesting problem. We begin by noting that
\begin{equation} \label{eq:integralfunc}
	\langle f, \varphi_\beta \rangle_{\mathcal{H}^2} = a_1 + \int_{1/2}^\infty \left(f(\sigma)-a_1\right)\left(\sigma-\frac{1}{2}\right)^{\beta-1}\,\frac{d\sigma}{\Gamma(\beta)}.
\end{equation}
The linear functional on $H^p(\mathbb{D})$ corresponding to \eqref{eq:integralfunc} is hence given by 
\begin{equation}\label{eq:boundaryint} 
	L_\beta(f) := \int_0^1 f(r) \left(1-r\right )^{\beta-1}\,\frac{dr}{\Gamma(\beta)}. 
\end{equation}
A computation with the Beta integral gives that $L_\beta(f) = \langle f, \psi_\beta \rangle_{H^2(\mathbb{D})}$ with 
\begin{equation} \label{eq:betadisc}
	\psi_\beta(z) = \sum_{j=0}^\infty \frac{\Gamma(j+1)}{\Gamma(j+1+\beta)} z^j.
\end{equation}
We note that $\Gamma(j+1)/\Gamma(j+1+\beta) \simeq_\beta (j+1)^{-\beta}$ and compile the following result: 
\begin{theorem}\label{thm:dualitydisc} 
	Let $\psi_\beta$ be as in \eqref{eq:betadisc}. Then 
	\begin{enumerate}
		\item[(a)] If $1<p<\infty$, then $\psi_\beta$ is in $\left(H^p(\mathbb{D})\right)^\ast = H^{p/(p-1)}(\mathbb{D})$ if and only if $\beta>1/p$. 
		\item[(b)] If $p\leq 1$, then $\psi_\beta$ is in $\left(H^p(\mathbb{D})\right)^\ast$ if and only if $\beta\geq1/p$. Moreover, if $\beta\geq1$, then $\psi_\beta$ is in $H^p(\mathbb{D})$ for every $p<\infty$. 
	\end{enumerate}
\end{theorem}
\begin{proof}
	We begin with (a). That $\left(H^p(\mathbb{D})\right)^\ast = H^{p/(p-1)}(\mathbb{D})$ for $1<p<\infty$ is well-known (see \cite{Duren}). We will investigate when $\psi_\beta$ is in $H^{p/(p-1)}(\mathbb{D})$. To do this, we use a result of Hardy and Littlewood \cite{HL31}: If $f(z) = \sum_{j=0}^\infty a_j z^j$ has positive and decreasing coefficients and $1<q<\infty$, then
	\[\|f\|_{H^q(\mathbb{D})} \simeq_q \left(\sum_{j=0}^\infty (j+1)^{q-2}a_j^q \right)^\frac{1}{q}.\]
	Setting $q=p/(p-1)$ we find that
	\[\|\psi_\beta\|_{H^q\mathbb{D})}^q \simeq_q \sum_{j=0}^\infty (j+1)^{\frac{p}{p-1}(1-\beta)-2},\]
	which is finite if and only if $\beta>1/p$.
	
	For (b), we begin with the case $\beta=1$. A stronger version of our statement can be found in \cite[Thm.~4.5]{Duren}. It is also clear that since $\psi_1$ is in $\left(H^1(\mathbb{D})\right)^\ast$, $\psi_1$ is in $H^p(\mathbb{D})$ for every $p<\infty$.
	
	To investigate the case $p<1$, we require the main result in \cite{DRS69} for which we refer to \cite{Duren}. We conclude that $\psi_\beta \in \left(H^p(\mathbb{D})\right)^\ast$ if and only if $\beta\leq 1/p$ by combining \cite[Thm.~7.5]{Duren} with \cite[Ex.~1 and Ex.~3 on p. 90]{Duren}. If $\beta<1$, then $\psi_\beta$ is a bounded function, so $\psi_\beta$ is in $H^p(\mathbb{D})$ for every $p<\infty$. 
\end{proof}

In analogy with Theorem~\ref{thm:dualitydisc}, we therefore offer the following conjecture. 
\begin{conj}
	Let $0<p\leq2$. The Dirichlet series $\varphi_{1/p}$ from \eqref{eq:betafunc} defines a bounded linear functional on $\mathcal{H}^p$ if and only if $0<p\leq1$. 
\end{conj}

One possible way to approach to this conjecture is to translate Theorem~\ref{thm:dualitydisc} to the Hardy spaces of the half-plane $\mathbb{C}_{1/2}$. For $0<p<\infty$, set
\[\|f\|_{H^p_{\operatorname{i}}(\mathbb{C}_{1/2})}^p := \frac{1}{\pi}\int_{-\infty}^\infty |f(1/2+it)|^p \,\frac{dt}{1+t^2}.\]
From Theorem~\ref{thm:dualitydisc} and a standard computation using a linear fractional mapping from $\mathbb{C}_{1/2}$ to $\mathbb{D}$, we find that the inequality
\[\left|\int_{1/2}^{3/2} f(\sigma) \,\left(\sigma-\frac{1}{2}\right)^{1/p-1}d\sigma\right| \leq C_p \|f\|_{H^p_{\operatorname{i}}(\mathbb{C}_{1/2})}\]
holds if and only if $0<p\leq 1$. 

One could hope to settle both the positive (for $0<p\leq1$) and negative (for $1<p\leq2$) part of the conjecture by relating $H^p_{\operatorname{i}}(\mathbb{C}_{1/2})$ to $\mathcal{H}^p$. The most direct approach along these lines would be to employ the embedding and interpolation results discussed in Section~2.5, respectively. 

However, a recent result by Harper \cite{Har2} shows that the embedding
\begin{equation} \label{eq:embedding}
	\|f\|_{H^p_{\operatorname{i}}(\mathbb{C}_{1/2})} \leq C_p \|f\|_{\mathcal{H}^p}
\end{equation}
does not hold for $0<p<2$. This means that a positive result in the range $0<p\leq1$ cannot be obtained by using \eqref{eq:embedding}. However, it should be noted that the result from \cite{Har2} does not yield any conclusion regarding our conjecture. Moreover, the fact that we only have interpolation results for $\mathcal{H}^{2/k}$ when $k=1,2,3,\ldots$ means we cannot extract the conjectured negative result in the range $1<p<2$ either. 

We end this section by establishing a weaker result, which may serve as a replacement for \eqref{eq:embedding} in certain settings (see \cite{BB16}). For $\alpha>1$, set
\[\|f\|_{A^2_{\alpha,\operatorname{i}}(\mathbb{C}_{1/2})} := \left(\int_{\mathbb{C}_{1/2}}|f(s)|^2\,(\alpha-1)\left(\sigma - \frac{1}{2}\right)^{\alpha-2}\,\frac{4^{\alpha-1}dm(s)}{\pi|s+1/2|^{2\alpha}}\right)^\frac{1}{2}.\]
The following result is an extension of \cite[Thm.~1]{BB16} from $1\leq p < 2$ to the full range $0<p<2$. 
\begin{corollary}\label{cor:embedding} 
	Let $0<p<2$. There is a constant $C_p\geq1$ such that
	\[\|f\|_{A^2_{2/p,\operatorname{i}}(\mathbb{C}_{1/2})} \leq C_p \|f\|_{\mathcal{H}^p}\]
	for every $f \in \mathcal{H}^p$. The parameter $\alpha=2/p$ is optimal. 
\end{corollary}
\begin{proof}
	Define $\mathcal{H}_\alpha$ as the Hilbert space of Dirichlet series $f(s) = \sum_{n=1}^\infty a_n n^{-s}$ that satisfy
	\[\|f\|_{\mathcal{H}_\alpha} := \left(\sum_{n=1}^\infty \frac{|a_n|^2}{\Phi_\alpha(n)}\right)^\frac{1}{2} < \infty.\]
	Here it is crucial that $\Phi_\alpha$ is strictly positive. By \eqref{eq:avgorder} and \cite[Thm.~1]{Olsen11} it follows that there is some $C_\alpha$ such that
	\[\|f\|_{A^2_{\alpha,\operatorname{i}}(\mathbb{C}_{1/2})} \leq C_\alpha \|f\|_{\mathcal{H}_\alpha},\]
	whenever $\alpha>1$. The proof of the first statement is completed using \eqref{eq:multmix2}. For the proof that $\alpha=2/p$ is optimal, we can follow the argument given in the proof of \cite[Thm.~1]{BB16}. We set
	\[f_{p,\varepsilon}(s) = \zeta^{2/p}(1/2+\varepsilon+s) = \sum_{n=1}^\infty \frac{d_{2/p}(n)}{n^{1/2+\varepsilon}}\,n^{-s},\]
	which satisfies
	\[f_{p,\varepsilon}(s) = \left(\frac{1}{1/2+\varepsilon+s-1}\right)^{2/p} + O\left( |1/2+\varepsilon+s-1|^{-(2/p-1)}\right)\]
	when $1/2 < \mre s=\sigma < 3/2$ and $0 < \mim s=t < 1$. Then clearly
	\[\|f_{p,\varepsilon}\|_{A^2_{\alpha,\operatorname{i}}(\mathbb{C}_{1/2})}^2 \gtrsim \int_{1/2}^{3/2} \int_{0}^1 \left|\frac{1}{\sigma-1/2+\varepsilon+it}\right|^\frac{4}{p}\,\left(\sigma-\frac{1}{2}\right)^{\alpha-2} \,dtd\sigma \gtrsim \varepsilon^{\alpha-4/p}.\]
	Since $\|f_{p,\varepsilon}\|_{\mathcal{H}^p}^2 \simeq \varepsilon^{-2/p}$, we get that $\alpha-4/p\geq-2/p$ is necessary. 
\end{proof}

\section{Coefficient estimates}\label{se:coeff} 
We now turn to weighted $\ell^\infty$ estimates of the coefficient sequence $(a_n)_{n\geq 1}$ for elements $f(s)=\sum_{n=1}^\infty a_n n^{-s}$ in $\Hp$. Phrased differently, we are interested in estimating the norm of the linear functional $f \mapsto a_n$ for every $n\ge 1$, i.e., the quantity
\[ \mathcal{C}(n,p):= \sup_{\|f\|_p=1} |a_n|. \]
When $p\ge 1$, $a_n$ can be expressed as a Fourier coefficient, implying that this norm is trivially $1$ for all $n$. We will therefore mainly be concerned with the case $0<p<1$.

Our first observation is that it suffices to deal with the one-dimensional situation because the general estimates will appear by multiplicative extension. Before we prove this claim, we recall what is known about the coefficients of $f(z)=\sum_{k=0}^\infty a_k z^k$ in $H^p(\D)$ when $0<p<1$. For $0<p<\infty$ and $k\geq1$, we set 
\begin{equation}\label{eq:cnp} 
	C(k,p):=\displaystyle \sup\left\{\frac{|f^{(k)}(0)|}{k!} \, : \, \| f\|_{H^p(\D)}=1\right\}. 
\end{equation}
By a classical result \cite[p.~98]{Duren}, $C(k,p)\lesssim k^{1/p-1}$ and $a_k =o(k^{1/p-1})$ for an individual function in $H^p(\D)$ when $0<p<1$. By a normal family argument, there are extremal functions $f_k$ in $H^p(\mathbb{D})$ for \eqref{eq:cnp}.

Turning to the multiplicative extension, we begin by noting that it suffices to consider an arbitrary polynomial
\[F(z)=\sum_{\kappa} c_\kappa z^{\kappa}\]
on $\T^\infty$ and to estimate the size of $c_\kappa$ for an arbitrary multi-index $\kappa=(\kappa_1,\ldots,\kappa_m, 0,0 ,\ldots)$. Recall that $A_m F$ denotes the $m$te Abschnitt of $F$. For $0<p<1$ we use \eqref{eq:03} to find that 
\begin{align*}
	|c_\kappa|^p & =\left| \int_{\T^m} A_m F(z) \overline{z_1}^{\kappa_1}\cdots \overline{z_m}^{\kappa_m} d\mu_m \right|^p \\
	&\le C(\kappa_m,p)^p \int_{\T} \left| \int_{\T^{m-1}} A_mF(z) \overline{z_1}^{\kappa_1}\cdots \overline{z_{m-1}}^{\kappa_{m-1}} d\mu_{m-1}\right|^p d\mu_1 \\
	& \le C(\kappa_1,p)^p \cdots C(\kappa_m,p)^p \| A_mF\|_p^p \\
	& \le C(\kappa_1,p)^p \cdots C(\kappa_m,p)^p \| F\|_p^p. 
\end{align*}
This is a best possible estimate because if $f_k$ in $H^p(\D)$ satisfies $|a_k|/\| f_k \|_p= C(k,p)$, then clearly the function
\[\prod_{j=1}^m f_{\kappa_j} (z_j) \]
will be extremal with respect to the multi-index $\kappa=(\kappa_1,\ldots,\kappa_m)$. Hence we conclude that $n\mapsto \mathcal{C}(n,p)$ is a multiplicative function that takes the value $C(k,p)$ at $n=p_j^k$ for every prime $p_j$.

To the best of our knowledge, the exact values of $C(k,p)$ from \eqref{eq:cnp} have not been computed previously for any $k\geq 1$ when $0<p<1$, and we have therefore made an effort to improve this situation. We begin with the case $k=1$ which is settled by the following theorem. Note that (b) below is a special case of \cite[Ex.~2 on p.~143]{Duren}.
\begin{theorem}\label{le:1} 
	We have 
	\begin{equation}\label{eq:cn1} 
		C(1,p)=1 \quad \text{if}\quad p\geq1,\quad\text{and}\quad C(1,p)=\sqrt{\frac{2}{p}}\left(1-\frac{p}{2}\right)^{\frac{1}{p}-\frac{1}{2}}\quad \text{if}\quad 0<p<1. 
	\end{equation}
	The corresponding extremals $($modulo the trivial modifications $f(z)\mapsto e^{i\theta_1}f(e^{i\theta_2}z))$ are 
	\begin{itemize}
		\item[(a)] $f(z)=z$ for $p>1$; 
		\item[(b)] the family $f_a(z)=\left(a+\sqrt{1-a^2}z\right)\left(\sqrt{1-a^2}+az\right)$ with $0\le a\le 1$ for $p=1$; 
		\item[(c)] $f(z)= \left(\sqrt{1-p/2}+z\sqrt{p/2}\right)^{2/p}$ for \ $0<p<1$. 
	\end{itemize}
\end{theorem}
\begin{proof}
	As already pointed out, it is obvious that $C(1,p)=1$ when $p\ge 1$. The uniqueness of the extremal function for $p>1$ is immediate by the strict convexity of the unit ball of $L^p(\T)$.
	
	To find the extremal functions when $p=1$, we start from the fact that functions $f$ in the unit ball of $H^1(\T)$ can equivalently be written in the form $f=gh$, where $h,g$ are in the unit ball of $H^2(\T)$. Writing $g(z)=\sum_{k=0}^\infty g_kz^k$ and $h(z)=\sum_{k=0}^\infty h_kz^k$, our task is to maximize
	\[ f'(0)=g_0h_1+g_1h_0, \]
	under the sole condition that $\sum_{k=0}^\infty |g_k|^2=1$ and $\sum_{k=0}^\infty |h_k|^2=1$. By the Cauchy--Schwarz inequality, we must have $|g_0|^2+|g_1|^2=|h_0|^2+|h_1|^2=1$ and also $(g_0,g_1)=\lambda (h_1,h_0)$ for a unimodular constant $\lambda$. We may choose $(g_0, g_1)$ as an arbitrary unit vector, and hence we get the stated extremals.
	
	We turn to the case $0<p<1$. By invoking the inner-outer factorization of $f$, we may write an arbitrary element $f$ in the unit ball of $H^p(\D)$ equivalently as $f= gh^{2/p-1}$, where $g, h$ are in the unit ball of $H^2(\T)$ and $h$ has no zeros in $\D$. We denote the coefficients of $g$ and $h$ as before. By applying a suitable transformation $f(z)\mapsto e^{i\theta_1}f(e^{i\theta_2}z)$, we may assume that $h_0,h_1\geq 0$, and, moreover, that $f(0)= g_0h_0^{2/p-1}$, where $h_0^{2/p-1}\geq 0$ is chosen to be real and nonnegative. Hence
	\[C(1,p)=\sup\left( h_0^{2/p-1}g_1+\left(\frac{2}{p}-1\right)h_0^{2/p-2}h_1g_0\right),\]
	where the supremum is over all pairs $(g_0,g_1)$ with $|g_0|^2+|g_1|^2=1$ and pairs of nonnegative numbers $(h_0,h_1)$ with $h_0^2+h_1^2=1$ and $h_0\ge h_1$ since $h$ is zero-free.
	
	The maximum occurs when $(g_0,g_1)$ is a multiple of 
	\[\left(\left(\frac{2}{p}-1\right)h_0^{2/p-2}h_1, h_0^{2/p-1}\right)\]
	and hence
	\[ C(1,p)^2=\max_{h_0^2+h_1^2=1, h_0\ge h_1\ge 0}\left(h_0^{4/p-2}+\left(\frac{2}{p}-1\right)^2h_0^{4/p-4}h^2_1\right). \]
	Suppressing the condition $h_0\ge h_1$, we find that 
	\begin{equation}\label{eq:extremal} 
		C(1,p)^2\leq \max_{x\in [0,1]} \left(x^{4/p-2}+\left(\frac{2}{p}-1\right)^2x^{4/p-4}\left(1-x^2\right)\right)= \frac{2}{p}\left(1-\frac{p}{2}\right)^{2/p-1} 
	\end{equation}
	by an elementary calculus argument. Since the solution to the extremal problem in \eqref{eq:extremal} corresponds to $h_0=\sqrt{1-p/2}$, we also have $h_0\geq h_1$, and the inequality sign in \eqref{eq:extremal} can therefore in fact be replaced by an equality sign. 
\end{proof}

For future reference, we notice that the following asymptotic estimates hold: 
\begin{equation}\label{eq:asympp} 
	C(1,p)=
	\begin{cases}
		1+(1-\log 2)(1-p) +O((1-p)^2), & p\nearrow 1 \\
		\frac{1}{\sqrt{p}} \left(\sqrt{2/e}+O(p)\right), & p\searrow 0. 
	\end{cases}
\end{equation}

For $k\ge 2$, the method used in the preceding proof will lead to a similar finite-dimensional extremal problem. The solution to this problem is plain for all $k\ge 2$ when $p=1$, but in 
the range $0<p<1$, the complexity increases notably with $k$, and we have made no attempt to deal with it. Instead, we will supply (non-optimal) estimates for which we require the following remarkable contractive estimate of Weissler \cite[Cor.~2.1]{Weissler80} for the dilations
\[f_r(z): = f(rz), \quad r>0, \]
of functions $f$ in $H^p(\D)$. 
\begin{lemma}\label{lem:weissler} 
	Let $0<p \leq q<\infty$. The contractive estimate
	\[\|f_r\|_{H^q(\mathbb{D})} \leq \|f\|_{H^p(\mathbb{D})}\]
	holds for every $f$ in $H^p(\mathbb{D})$ if and only if $r \leq \sqrt{p/q}$. 
\end{lemma}
The desired estimates can now be obtained from the Cauchy integral formula and Lemma~\ref{lem:weissler}. 
\begin{lemma}\label{lem:gen} 
	Suppose that $0<p<1$ and $k\geq 1.$ Then
	\[C(k,p)\le \min_{p\le x < 1} x^{-k/2} (1-x)^{1/x-1/p}. \]
\end{lemma}
\begin{proof}
	Suppose that $f(z)=\sum_{k=0}^\infty c_k z^k$ is in $H^p(\D)$ with $\| f\|_p=1$. Then, by Cauchy's formula,
	\[ |c_k|\leq \frac{1}{2\pi r}\int_{|z|=r}|z^{-k}f(z)|\,|dz| \]
	for every $0<r<1$. Using the pointwise estimate $|f(z)|\leq \left(1-|z|^2\right)^{-1/p}\|f\|_p$, we therefore find that
	\[ |c_k|\leq r^{-k}\left((1-r^2)^{1/p}\right)^{q-1} \| f\|_p^{1-q} \int_{\T}|f(rz)|^q d\mu(z) \]
	whenever $0<r<1$ and $0<q<1$. Choosing $p<q\le 1$ and $r^2=p/q$ and invoking Lemma~\ref{lem:weissler}, we obtain the desired result. 
\end{proof}


We will now use the information gathered above to prove a result about the maximal order of the multiplicative function $n\mapsto \mathcal{C}(n,p)$. To begin with, we notice that, by
Theorem~\ref{le:1},
	\[ \mathcal{C}(n,p)= C(1,p)^{\omega(n)} \]
	when $n$ is a square-free number and hence
	\begin{equation}\label{eq:Csf} \limsup_{\substack{n\to\infty \\ \mu(n)\neq0}} 
	\frac{\log \mathcal{C}(n,p)}{\log n/\log\log n}=\log C(1,p)
	 \end{equation}
since
\[ \limsup_{\substack{n\to \infty \\ \mu(n)\neq0}} \frac{\log \omega(n)}{\log n/\log\log n} =1. \]
It seems reasonable to expect that the $\limsup$ in \eqref{eq:Csf} is unchanged if we drop the restriction that $\mu(n)\neq 0$. The next theorem is as close as we have been able to get to confirming this conjecture,
based on our general bounds for $C(k,p)$. 
\begin{theorem}\label{th:growth} 
	Assume that $0<p<1$. Then
	\[ 0<\limsup_{n\to \infty} \frac{\log \mathcal{C}(n,p)}{\log n/\log\log n} < \infty. \]
	Moreover,
	\[ \limsup_{n\to \infty} \frac{\log \mathcal{C}(n,p)}{\log n/\log\log n}=
	\begin{cases}
		\frac{1}{2}|\log p|(1+O(p)), & p\searrow 0, \\
		c_p(1-p), & p\nearrow 1, 
	\end{cases}
	\]
	where $1-\log 2+O(1-p)\le c_p \le 1/2+O(1-p)$. 
\end{theorem}
\begin{proof}
	The general lower bound for the $\limsup$ follows from \eqref{eq:Csf}, while the 
	lower bounds
	\[ \limsup_{n\to \infty} \frac{\log \mathcal{C}(n,p)}{\log n/\log\log n}\ge 
	\begin{cases}
		\frac{1}{2}|\log p|(1+O(p)), & p\searrow 0, \\
		(1-\log 2)(1-p)+O((1-p)^2), & p\nearrow 1, \end{cases} \]
	follow from \eqref{eq:Csf} along with \eqref{eq:asympp}. To get an upper bound for the $\limsup$ when $p\searrow 0$, we use first H\"older's inequality and then \eqref{eq:multmix2} to obtain
	\[
	 \mathcal{C}(n,p)\leq  \mathcal{C}\left(n,\frac{2}{\lceil 2/p\rceil}\right)\leq \sqrt{\Phi_{\lceil 2/p\rceil}(n)}
	 = \sqrt{d_{\lceil 2/p\rceil}(n)}.
	\]
	By combining this with the fact
	\[ \limsup_{n\to \infty} \frac{\log d_{\alpha}(n)}{\log n/\log\log n} =\log \alpha. \]
	we obtain the general upper bound $\frac{1}{2} \log( \lceil 2/p\rceil)$, which give the desired upper bound as 
	$p\searrow 0.$ We note that the reason for switching to the integer value $\lceil 2/p\rceil$ was that
	for non-integral values of $\alpha$ the quantity \eqref{eq:multmixweight} contains an exponentially growing factor in $\Omega (n)$.
	
	To get an upper bound for the $\limsup$ when $p\nearrow 1$, we argue as follows. Set
	\[ n=\prod_{j} p_j^{\kappa_j}. \]
	For $\kappa_j\le 1/(1-p)$, we set $x=p$ in Lemma~\ref{lem:gen} and get
	\begin{equation} \label{eq:Csmall} C(\kappa_j,p)
	\le p^{-\kappa_j/2}. \end{equation}
	We note that
	\begin{align} \nonumber \sum_{ \kappa_j\le 1/(1-p)} \kappa_j & \le 
	\frac{1}{1-p} \sum_{j\le \log n/(\log\log n)^2} 1
	+\frac{(1+o(1))}{\log\log n} \sum_{ \kappa_j\le 1/(1-p)} \kappa_j \log p_j \\
	&\label{eq:smallkap} =  \frac{\log n}{(1-p)(\log\log n)^2} + \frac{(1+o(1))}{\log\log n}\sum_{ \kappa_j\le 1/(1-p)} \kappa_j \log p_j . \end{align}
	For $\kappa_j> 1/(1-p)$, we set $x=1-(1-p)/\kappa_j$ in Lemma~\ref{lem:gen} so that
	\begin{equation}\label{eq:Cbig} C(\kappa_j,p)
	\le \left(1-\frac{(1-p)}{\kappa_j}\right)^{-\kappa_j/2} \left(\frac{1-p}{\kappa_j}\right)^{1-1/p}
	\le e^{1-p} \kappa_j^{2(1/p-1)}. \end{equation}
	We observe that, given $\varepsilon>0$, we will have if $p$ is close enough to 1, then
	\begin{align} \nonumber  \sum_{ \kappa_j\ge 1/(1-p)} \log \kappa_j & \le 
	\log(\log n/\log 2) \sum_{j\le \log n/(\log\log n)^3} 1 + \frac{\varepsilon}{\log\log n}
	\sum_{\kappa_j>1/(1-p)}\kappa_j \log p_j \\
	& \label{eq:bigk} \le  \frac{(1+o(1))\log n}{(\log\log n)^2}+ \frac{\varepsilon}{ \log\log n}\sum_{ \kappa_j> 1/(1-p)} \kappa_j \log p_j\end{align}
	if $p$ is close enough to 1. Putting \eqref{eq:smallkap} into \eqref{eq:Csmall} and \eqref{eq:bigk} into
	\eqref{eq:Cbig}, respectively, we find that
\begin{align*} 
\log \mathcal{C}(n,p)& =\sum_{j} \log C(\kappa_j, p) 
=\sum_{ \kappa_j \le 1/(1-p)} C(\kappa_j, p) +\sum_{ \kappa_j > 1/(1-p)} C(\kappa_j, p) \\
& \le o(1) \frac{\log n}{\log\log n} +\frac{|\log p|}{2\log\log n} \sum_{ \kappa_j \le 1/(1-p)} \kappa_j \log p_j
+\frac{2(1/p-1)\varepsilon}{\log\log n} \sum_{ \kappa_j > 1/(1-p)} \kappa_j \log p_j
\end{align*}
holds for arbitrary $\varepsilon>0$, if $p$ is close enough to $1$. Choosing  $\varepsilon<1/4$, we obtain the desired upper bound for the $\limsup$.
\end{proof}

\section{Estimates for the partial sum operator}\label{se:partial}

Assume that $f(s)=\sum_{n=1}^\infty a_n n^{-s}$ is a Dirichlet series in $\Hp$ for some $p>0$. For given $N\geq 1$, the partial sum operator $S_N$ is defined as the map
\[S_N\left(\sum_{n=1}^\infty a_n n^{-s}\right) := \sum_{n=1}^N a_n n^{-s}.\]
It is of obvious interest to try to determine the norm of $S_N$ when it acts on the Hardy spaces $\Hp$. Helson's version of the M. Riesz theorem \cite{H1} shows that $S_N$ is bounded for $1<p<\infty$, and, moreover, its norm is bounded by the norm of the one-dimensional Riesz projection acting on functions in $H^p(\D)$. Furthermore, by the same argument of Helson \cite{H1}, we have the following. 
\begin{lemma}\label{lem:partial} 
	Suppose that $0<p<1$. We have the estimate
	\[ \| S_N f\|_{\Hp} \le \frac{A}{(1-p)} \| f\|_{\Ho} \]
	for $f$ in $\Hp$, where $A$ is an absolute constant. 
\end{lemma}

We refer to \cite[Sec.~3]{AOS}, where it is explained how the lemma follows from Helson's general result concerning compact Abelian groups whose dual is an ordered group \cite{H1}. See also Sections~8.7.2 and 8.7.6 of \cite{R}. In our case, the dual group in question is the multiplicative group of positive rational numbers $\mathbb{Q}_+$ which is ordered by the numerical size of its elements. This means that the bound for $\| S_N\|_{\Hp\to \Hp}$ in the range $1<p<\infty$ relies on the additive structure of the positive integers.

When $0<p\le 1$ or $p=\infty$, a natural question is to determine the asymptotic growth of the norm $\|S_N\|_{\Hp\to\Hp}$ when $N\to\infty$. It is known from \cite{BCQ06} and \cite{BQS} that the growth of both $\| S_N\|_{\Ho\to\Ho}$ and $\|S_N \|_{\Hi\to \Hi}$ is of an order lying between $\log\log N$ and $\log N$. We will confine our discussion to the range $0<p\le 1$ and begin with a new result for the case $p=1$. 
\begin{theorem}\label{th:norm1} 
	We have
	\[ \log\log N \lesssim \| S_N \|_{\Ho\to\Ho} \lesssim \frac{\log N}{\log \log N}. \]
\end{theorem}
\begin{proof}
	Using H\"{o}lder's inequality with $p=(1+\varepsilon)/\varepsilon$ and $p'=1+\varepsilon $, we get
	\[ \| g\|_{\Ho}^{1-\varepsilon} \le \left(\frac{\| g \|_{\Ht}}{\|g\|_{\Ho}}\right)^{2\varepsilon } \| g \|_{ \mathcal{H}^{1-\varepsilon}}^{1-\varepsilon}. \]
	Setting $g=S_Nf$ and applying Lemma~\ref{lem:partial}, we get
	\[ \| S_N f \|_1 \le A \frac{1}{\varepsilon} \left(\frac{\| S_N f \|_2}{\|S_N f\|_1}\right)^{2\varepsilon/(1-\varepsilon) } \| f \|_1. \]
	Now we need to understand how large the ratio $\| f \|_2/\| f \|_1$ can be when $f$ is a Dirichlet polynomial of length $N$. A precise solution to this problem can be found in the recent paper \cite{DP}. For our purpose, the following one-line argument suffices. By Helson's inequality (which is \eqref{eq:multmix2} for $p=1$) and a well-known estimate for the divisor function, we have
	\[ \| f \|_2 \le \max_{n\le N} \sqrt{d(n)} \|f\|_1 \le 2^{\left(\frac{1}{2}+o(1)\right)\frac{\log N}{\log\log N}} \|f\|_1. \]
	This means that we can choose $\varepsilon=(\log\log N)/\log N$ so that we get
	\[ \| S_N \|_1 \lesssim (\log N)/\log\log N , \]
	as desired.
	
	The lower bound is obvious from the classical one-dimensional result: The Bohr lift maps Dirichlet series in $\Hp$ of the form $\sum_{k=0}^\infty c_k 2^{-ks}$ to functions in $H^p(\D)$. 
\end{proof}

It is interesting to notice that our improved upper bound relies on both an additive argument (Lemma~\ref{lem:partial}) and a multiplicative argument (Lemma~\ref{lem:multmixedineqs}). We now turn to the case $0<p<1$ which will again require a mixture of additive and multiplicative arguments. 
\begin{theorem}\label{thm:smallp} 
	Suppose that $0<p<1$. There are positive constants $\alpha_p\le \beta_p$ such that
	\[ e^{\alpha_p \frac{\log N}{\log\log N}} \lesssim \| S_N\|_{\Hp\to\Hp} \lesssim e^{\beta_p \frac{\log N}{\log\log N}}. \]
	Moreover, we have
	\[ \liminf_{N\to \infty}\frac{\log \| S_N\|_{\Hp\to\Hp}}{\log N/\log\log N}\ge
	\begin{cases}
		\frac{1}{4}|\log p|+O(1), &p\searrow 0 \\
		\frac{1}{2}(1-\log 2)(1-p)+O((1-p)^2),& p\nearrow 1 
	\end{cases}
	\]
	and
	\[ \limsup_{N\to \infty}\frac{\log \| S_N\|_{\Hp\to\Hp}}{\log N/\log\log N}\le 
	\begin{cases}
		\frac{1}{2}|\log p|+O(1), &p\searrow 0 \\
		c (1-p) +O((1-p)^2),& p\nearrow 1,
	\end{cases}
	\]
	where $c$ is an absolute constant. 
\end{theorem}
We have made no effort to minimize the constant $c$, but mention that our proof gives the value $\log 2$ times the norm of the operator $f\mapsto f^*$ from $H^1(\D)$ to $L^1(\T)$, where $f^*$ is the radial maximal function of $f$. Comparing with Theorem~\ref{th:growth}, we notice that $\log \| S_N\|_{\Hp\to\Hp}$ has essentially the same maximal order as that of $\log \mathcal{C}(N,p)$.

We will split the proof of Theorem~\ref{thm:smallp} into three parts. We begin with the easiest case, where we apply Lemma~\ref{lem:multmixedineqs} in the same way as the first part of the proof of Theorem~\ref{th:growth}.
\begin{proof}
	[Proof of the upper bound in Theorem~\ref{thm:smallp}, with the asymptotics when $p\searrow 0$] We begin by setting $\alpha:=\lceil 2/p \rceil$ and apply the Hardy--Littlewood inequality from Lemma~\ref{lem:multmixedineqs}:
	\[ \| S_N f \|_{\mathcal{H}^p} \le \| S_N f \|_{\mathcal{H}^2} \le \left(\max_{n\le N} \sqrt{d_{\alpha}(n)}\right)\left(\sum_{n=1}^\infty \frac{|a_n|^2}{d_{\alpha} (n)}\right)^\frac{1}{2} \lesssim \alpha^{\frac{\log{N}}{2\log\log{N}}(1+o(1))}\|f\|_{\mathcal{H}^{2/\alpha}},\]
	where we in the last step used that
	\[d_\alpha(n) \leq \alpha^{\frac{\log{n}}{\log\log{n}}(1+o(1))}\]
	when $n\to\infty$. We conclude by using that $\|f\|_{\mathcal{H}^{2/\alpha}}\leq\|f\|_{\mathcal{H}^p}$, which holds because $2/\alpha\le p$. This argument gives both $\beta_p=\log(1+2/p)$, say, and the desired asymptotic estimate when $p\searrow 0$.
\end{proof}

We need a more elaborate argument to get the right asymptotic behavior when $p\nearrow 1$. We prepare for the proof by first establishing an auxiliary result concerning polynomials on $\T$. Here we use again the notation $f_r(z):=f(rz)$, where $f$ is an analytic function on $\D$ and $r>0$. 
\begin{lemma}\label{lem:bernstein} 
	Suppose that $0<p\le 1$. There exists an absolute constant $C$, independent of $p$, such that if \ $1-r=C^{-1/p}n^{-1}$, then 
	\begin{equation}\label{eq:basic} 
		\|Q\|_{H^p(\D)}^p\le 2\| Q_r\|_{H^p(\D)}^p
	\end{equation}
	for every polynomial $Q(z)=\sum_{k=0}^n c_k z^k$. 
\end{lemma}
\begin{proof}
	For this proof, we write $\| Q\|_p=\|Q\|_{H^p(\D)}$. By the triangle inequality for the $L^p$ quasi-metric, we have 
	\begin{equation}\label{eq:back} 
		\| Q\|_p^p\le \| Q-Q_r \|_p^p+\|Q_r\|_p^p. 
	\end{equation}
	Since
	\[ |Q(z)-Q_r(z)|=\left|\int_{rz}^{z} Q'(w) dw \right|\le (1-r) \max_{0\le \rho\le 1} |Q'(\rho z)|, \]
	we find that
	\[ \| Q-Q_r \|_p^p\le A(1-r)^p \| Q'\|_p^p \]
	for an absolute constant $A$ by the $H^p$ boundedness of the radial maximal function. Using Bernstein's inequality for $0<p \le 1$ \cite{Ari, GL}, we therefore get that
	\[ \| Q-Q_r \|_p^p\le A(1-r)^p n^{p} \|Q\|_p^p.\]
	Returning to \eqref{eq:back}, we see that we get the desired result by setting $C=2A$. 
\end{proof}
\begin{proof}
	[Proof of the upper bound in Theorem~\ref{thm:smallp} when $p\nearrow 1$] Set
	\[ m=m(N):=\left\lfloor\frac{\log N}{(\log\log N)^3}\right\rfloor \]
	and write $z:=(u,v)$ for a point on $\T^\infty$, where $u=(z_1,...,z_m)$ and $v=(z_{m+1}, z_{m+2},...)$, so that $u$ corresponds to the first $m$ primes. Let $\xi$ and $\eta$ be complex numbers and set $\xi u:=(\xi z_1,...\xi z_m)$ and $\eta v=(\eta z_{m+1}, \eta z_{m+2},...)$. Also, if $F$ is a function on $\T^\infty$ and $0<r,\rho \le 1$, we set $F_{r,\rho}(z):=F(ru, \rho v)$. 
	
	We will now apply Lemma~\ref{lem:bernstein} in two different ways. We begin by applying it to the function $\xi \mapsto (\mathcal{B}S_N f)(\xi u, v)$, which is a polynomial of degree at most $\log N/\log 2$. This gives
	\[ \int_{\T} |\mathcal{B}S_N f (\xi u, v)|^pd\mu(\xi) \le 2 \int_{\T} | \mathcal{B}S_N f(r\xi u, v)|^p d\mu(\xi) \]
	for every point $(u,v)$ and hence
	\[ \|\mathcal{B}S_N f \|_p^p\le 2 \| \mathcal{B}S_N f_{r,1}\|_p^p \]
	by Fubini's theorem, with $1-r=C^{-1/p}(\log N/\log 2)^{-1}$. Next, we apply \eqref{eq:basic} to the function $\eta \mapsto (\mathcal{B}S_N f_{r,1})( u, \eta v)$, which is a polynomial of degree at most $(1+o(1)) \log N/\log\log N$. Hence we find that
	\[ \| \mathcal{B}S_N f \|_p^p\le 2^2 \| \mathcal{B}S_N f_{r,\rho}\|_p^p \]
	with $1-\rho=C^{-1/p}(1+o(1))\log\log N/\log N$. Applying \eqref{eq:basic} $k$ times in this way, we therefore get that 
	\begin{equation}\label{eq:start} 
		\| \mathcal{B}S_N f \|_p^p\le 2^{k+1} \| \mathcal{B}S_N f_{r,\rho^k}\|_p^p. 
	\end{equation}
	We choose $k$ such that $\rho^k\le \sqrt{p}$, which is done because our plan is to use Lemma~\ref{lem:weissler} (Weissler's inequality). Since $1-\rho=C^{-1/p}(1+o(1))\log\log N/\log N$, we therefore obtain the requirement that 
	\begin{equation}\label{eq:k} 
		k= \frac{\log p}{2 \log \rho}=|\log p|\cdot\left(1/2+o(1)\right)C^{1/p}\log N/\log \log N. 
	\end{equation}
	
	We now apply Lemma~\ref{lem:partial} to the right-hand side of \eqref{eq:start}, which yields
	\[ \| \mathcal{B}S_N f \|_p\le K(k,p) \| \mathcal{B} f_{r,\rho^k}\|_1, \]
	where
	\begin{equation}\label{eq:K} K(k,p):=A 2^{(k+1)/p} (1-p)^{-1}= A (1-p)^{-1} \exp\left(\left(\frac{\log 2}{2}+o(1)\right) |\log p| p^{-1} C^{1/p} \frac{\log N}{\log\log N}\right); \end{equation}
	here we took into account \eqref{eq:k} to get to the final bound for $K(k,p)$.
	Note that, in view of \eqref{eq:03}, we may assume that $v$ is a vector of length $d:=\pi(N)-m$. It follows that 
	\begin{align*}
		\| \mathcal{B}S_N f \|_p & \le K(k,p) \int_{\T^d} \int_{\T^m} | \mathcal{B} f(ru,\rho^k v) | d\mu_m(u) d\mu_{d}(v) \\
		& \le K(k,p) (1-r^2)^{-m(1-p)/p} \int_{\T^d} \left(\int_{\T^m} | \mathcal{B} f(u,\rho^k v) |^p d\mu_m(u)\right)^{1/p} d\mu_{d}(v), 
	\end{align*}
	where we in the last step used the Cole--Gamelin estimate \eqref{eq:CGestimate}. Using Minkowski's inequality
	\[\left(\int_X\left(\int_Y |g(x,y)|\,dy\right)^r\,dx\right)^\frac{1}{r} \leq \int_Y \left(\int_X |g(x,y)|^r\,dx\right)^\frac{1}{r}\,dy\]
	with $X = \mathbb{T}^d$ and $Y = \mathbb{T}^m$ and $r=1/p>1$, we thus get
	\[ \| \mathcal{B}S_N f \|_p^p \le K(k,p)^p (1-r^2)^{-m(1-p)} \int_{\T^m} \left(\int_{\T^d} | \mathcal{B} f(u,\rho^k v) | d\mu_d(v)\right)^{p} d\mu_{m}(u). \]
	We now iterate Weissler's inequality along with Minkowski's inequality $d$ times in the same way as in the proof of Lemma~\ref{lem:multmixedineqs} and get the bound
	\[ \| \mathcal{B}S_N f \|_p \le K(k,p) (1-r^2)^{-m(1-p)/p} \| f\|_p. \]
	Now taking into account our choice of $r$ and $m$, we find that 
\[ \limsup_{N\to \infty}\frac{\log \| S_N\|_{\Hp\to\Hp}}{\log N/\log\log N}\le 
		\limsup_{N\to \infty} \frac{\log K(k,p)}{\log N/\log\log N}. \]
Using finally \eqref{eq:K}, we conclude that  
\[ \limsup_{N\to \infty}\frac{\log \| S_N\|_{\Hp\to\Hp}}{\log N/\log\log N}\le 
		\frac{C^{1/p}|\log p| \log 2}{2p}, \]
and hence we get the desired asymptotics when $p\nearrow 1$ with $c=(C\log 2)/2$.
\end{proof}
\begin{proof}
	[Proof of the lower bound in Theorem~\ref{thm:smallp}] We consider first the special case when $M$ is the product of the first $k$ prime numbers, $M=p_1\cdots p_k$. By the prime number theorem, we have $k =(1+o(1)) \log{M}/\log\log{M}$. We use then the function
	\[f_M(s):= \prod_{j=1}^k \left(\sqrt{1-p_j/2}+p_j^{-s}\sqrt{p_j/2}\right)^{2/p}.\]
	We recognize each of the factors of this product as the extremal function from Theorem~\ref{le:1}. Hence $\|f_M\|_p=1$ and
	\[ f_M(s)=\sum_{n=1}^\infty a_n n^{-s} \]
	with
	\[ a_M=C(1,p)^k=\left(\sqrt{\frac{2}{p}}\left(1-\frac{p}{2}\right)^{\frac{1}{p}-\frac{1}{2}}\right)^k. \]
	Consequently, by the triangle inequality for the $L^p$ quasi-metric, 
	\begin{equation}\label{eq:bigN} 
		C(1,p)^{pk} \le \| S_{M-1} f_M\|_p^p+\| S_{M} f_M\|_p^p\le 2 \max\left(\| S_{M-1} f_M\|_p^p,\| S_{M} f_M\|_p^p\right), 
	\end{equation}
	and therefore at least one of the quasi-norms $\| S_{M-1} f_M\|_p$ or $\| S_{M} f_M\|_p$ is bounded below by
	\[ 2^{-1/p}C(1,p)^{(1+o(1)) \frac{\log M}{\log\log M}}. \]
	
	Suppose now that an arbitrary $N$ is given. Set $n_j:=p_1\cdots p_j$ and
	\[ J:=\max \{ j: \ N/n_j\ge n_j+1\}. \]
	It follows that $\log n_J=(1/2+o(1))\log N$. There are now two cases to consider: 
	\begin{enumerate}
		\item Suppose $\| S_{n_J} f_{n_J}\|_p$ is large. We set $x_N:=\lfloor N/n_J\rfloor$ and define
		\[ g_N(s):=x_N^{-s}f_{n_J} (s).\]
		Then $(S_N g_N)(s)=x_N^{-s} (S_{n_J}f_{n_J})(s)$ because $x_N=N/n_J-\varepsilon$ for some $0\le \varepsilon <1$, and so
		\[ x_N(n_J+1)=(N/n_J-\varepsilon)(n_J+1)=N+N/n_J-\varepsilon(n_J+1)>N,\]
		where we in the last step used the definition of $J$. 
		\item Suppose $\| S_{n_J-1} f_{n_J}\|_p$ is large. We set $x_N:=\lceil N/n_J \rceil$ and define $g_N$ as in the first case.
		Then $(S_N g_N)(s)=x_N^{-s} (S_{n_J-1}f_{n_J})(s)$ because $x_N=N/n_J+\varepsilon$ for some $0\le \varepsilon <1$, and so
		\[ x_N (n_J-1)=(N/n_J+\varepsilon)(n_J-1)=N-N/n_J+\varepsilon(n_J-1)<N,\]
		where we in the last step again used the definition of $J$. 
	\end{enumerate}
	
	In either case, since \eqref{eq:bigN} holds for $M=n_J$ and $\log n_J=(1/2+o(1))\log N$, we conclude that
	\[ \liminf_{N\to \infty}\frac{\log \| S_N\|_{\Hp\to\Hp}}{\log N/\log\log N}\ge \frac{1}{2} \log C(1,p). \]
	The proof is finished by invoking the asymptotic estimate \eqref{eq:asympp}. 
\end{proof}

Up to the precise values of $\alpha_p$ and $\beta_p$, the problem of estimating $\|S_N\|_{\Hp\to \Hp}$ for $0<p<1$ is solved by Theorem~\ref{thm:smallp}. This result is, however, somewhat deceptive because it is of no help when we need to estimate $\| S_N f\|_{\Hp}$ for functions $f$ of number theoretic interest such as those studied in \cite{BBSSZ}. In fact, in such cases, Lemma~\ref{lem:partial} may give a much better bound. 

\bibliographystyle{amsplain} 
\bibliography{hardyp}

\end{document}